\documentclass[lualatex,11pt]{article}

\usepackage{amssymb}
\usepackage{amsthm}
\usepackage{float}
\usepackage{tikz,pgf}

\newtheorem{thm}{Theorem}
\newtheorem{lem}{Lemma}
\newtheorem{rmk}{Remark}

\usepackage{graphicx}	
\begin{document}

\begin{center}
 {\large A quotient-lifting approach to the Hamiltonicity of the cylindrical 5-puzzle graph}

\bigskip
 Taizo Sadahiro
\end{center}

\begin{abstract}
In this short note, we construct an explicit Hamiltonian cycle in the state graph of the 5-puzzle on a cylindrical 
$2\times 3$ grid, a graph with $720$ vertices. 
The cycle is described by a short symbolic sequence of $48$ moves over the alphabet 
$\{{\mathtt L}, {\mathtt R}, {\mathtt V}\}$, repeated $15$ times, 
which can be verified directly. 
We also find a shorter 24-move sequence whose repetition yields a 2-cycle cover, 
which can be spliced into a Hamiltonian path. These constructions arise naturally 
from a general method: lifting Hamiltonian cycles from a quotient graph 
under the action of the puzzle’s symmetry group.
 The method produces compact, human-readable cycle encoding.

\end{abstract}

\section{Introduction}

The state graph of the classical sliding block puzzle, formed by 
placing $mn-1$ numbered tiles and one blank on a $m \times n$
rectangular grid, has long been a natural object of study 
in combinatorics and computer science. 
For $m=2$ and $n=3$, 
when the grid is modified to allow the wraparound move
(i.e., identifying opposite edges horizontally), 
we obtain a \emph{cylindrical} $5$-puzzle, whose state graph becomes 
a {\em cubic} graph on $720$ vertices representing all reachable 
configurations under legal moves. 
\footnote{A previous study of the classical (non-cylindrical) 5-puzzle 
appeared in~\cite{hanaokasadahiro23}, where different geometric structure
was uncovered.}

In this note, we construct an explicit Hamiltonian cycle 
in this graph, 
encoded by a short symbolic sequence over the move 
alphabet $\{{\mathtt L}, {\mathtt R}, {\mathtt V}\}$, 
where ${\mathtt L}, {\mathtt R}$, and ${\mathtt V}$ 
denote left, right, and vertical shifts of the blank tile. 
Remarkably, a $48$-move word, repeated $15$ times, 
traces a Hamiltonian cycle through all 720 configurations. 
We also identify a $24$-move word whose repetition yields a 
$2$-cycle cover, which can be spliced into a Hamiltonian path by 
simple edge replacements. 
As we show later, this symbolic sequence is so simple
that one can memorize and verifiable without the help
of computers. 

Our construction emerges from a general technique: we first reduce the state graph by symmetry to a smaller \emph{quotient graph}, where backtracking search can efficiently find a Hamiltonian cycle. We then \emph{lift} this cycle to the full state graph. This quotient--lifting method yields compressed symbolic encodings and appears to be effective in a variety of highly symmetric state graphs, suggesting the presence of a combinatorial ``Hamiltonian grammar.''
Although similar ideas may be implicitly used in related contexts, 
the author is not aware of previous work that formalizes 
this quotient-lifting approach in the context of Hamiltonian cycle construction.

%

The question of whether a given highly symmetric graph has a Hamiltonian cycle 
is a recurring theme in graph theory, exemplified by the Lovász conjecture
 on vertex-transitive graphs \cite{lovasz1969combinatorial}. 
While our focus lies on Cayley graphs arising from a puzzle,
these examples continue the exploration of 
Hamiltonicity in  meaningful settings.

\section{Preliminaries}

Let $\Gamma$ be a finite undirected simple graph with the vertex set $V(\Gamma)$ 
of cardinality $n+1$ and 
the edge set $E(\Gamma)$. 
A {\em position} of the {\em sliding block puzzle}
defined on $\Gamma$ is a bijection $f: V(\Gamma)\to \{1,2,\ldots, n, 0\}$.
We say $v\in V(\Gamma)$ is {\em blank} or {\em unoccupied} in
a position $f$ if $f(v)=0$, and hence, 
the vertex $f^{-1}(0)$ is blank.
A position $f$ is transformed into another position $g$  by a
{\em move} if there exist two mutually
adjacent vertices $v,w$ in $\Gamma$, such that, $v$ or $w$
is blank in $f$, and
\[
 g = f\circ (v,w),
\]
where $(v,w)$ denotes the transposition of $v$ and $w$.

Let ${\rm puz}(\Gamma)$ be the graph whose vertex set consists of
the positions of the puzzle.
Two vertices (or positions) $f$ and $g$ of ${\rm puz}(\Gamma)$
are connected by an edge $\{f,g\}$
if and only if there is a move transforming $f$ into $g$.

Let $H$ be a finite group and $S\subset H$ 
be its generating set. 
Then the Cayley graph ${\rm Cay}(H, S)$ is defined as follows:
The Cayley graph ${\rm Cay}(H, S)$ is a directed graph whose vertex set is $H$ 
and there is an edge from $h_1\in H$ to $h_2\in H$
if and only if there is some $s\in S$ such that $h_1s=h_2$.

\section{State graph as a Cayley graph}

\begin{figure}[t]
\centering
\begin{tikzpicture}[scale=1.5, every node/.style={circle,draw,inner sep=1.5pt}]
  \foreach \i in {0,1}{
    \foreach \j in {0,1,2}{
      \node[inner sep = 0.1] (v\i\j) at (\j,-\i) {(\i,\j)};
    }
  }
  \foreach \x in {0,1}{
    \draw (v\x0)--(v\x1);
    \draw (v\x1)--(v\x2);
  }
  \foreach \y in {0,1,2}{
    \draw (v0\y) -- (v1\y);
  }
 \draw[out=120, in=60] (v00) to (v02);
 \draw[out=-120, in=-60] (v10) to (v12);

\end{tikzpicture}
\caption{The cylindrical grid $\Gamma_{2,3}$. }
\end{figure}

In this note, we consider the sliding block puzzle defined
on the $2$-by-$3$ cylindrical grid graph $\Gamma_{2,3}$, which
is defined as follows.
The vertex set $V(\Gamma_{2,3})$ is ${\mathbb Z}/2{\mathbb Z}\times {\mathbb Z}/3{\mathbb Z}$,
and there is exactly one edge connecting $v$ and $w\in {\mathbb Z}/2{\mathbb Z}\times {\mathbb Z}/3{\mathbb Z}$,
if and only if,
\[
 v-w \in \left\{\pm(1,0), \pm(0,1)\right\} \in {\mathbb Z}/2{\mathbb Z}\times {\mathbb Z}/3{\mathbb Z}.
\]
We call this sliding block puzzle, the {\em cylindrical $5$-puzzle}.

Let $f$ be a position in the cylindrical $5$-puzzle,
which has blank  at $b=(x,y)\in {\mathbb Z}/2{\mathbb Z}\times {\mathbb Z}/3{\mathbb Z}$.
Then we obtain a permutation
\[
\sigma_f = 
 f((0,1)+b) 
 f((0,2)+b) 
 f((1,0)+b) 
 f((1,1)+b) 
 f((1,2)+b) \in S_{5},
\]
where $S_5$ denotes the symmetric group of degree five,
which we define to be the $S_5$-{\em part} of $f$.
Since $\Gamma_{2,3}$ is a non-bipartite graph,
by applying the Wilson's celebrated theorem on general
sliding block puzzle \cite{wilson},
we see that the {\em state space} is
\[
 G=S_5\times{\mathbb Z}/2{\mathbb Z}\times{\mathbb Z}/3{\mathbb Z}.
\]
The $\left({\mathbb Z}/2{\mathbb Z}\times{\mathbb Z}/3{\mathbb Z}\right)$-part
determines the blank position
and $S_5$-part the relative configuration.
Thus each position $f$ can be uniquely encoded as an
element $(\sigma_f, x_f, y_f)\in G$.
We sometimes identify $f$ with its encoding $(\sigma_f, x_f, y_f)\in G$.

Let the {\em right move} be the move which
exchanges the blank vertex $b$ with $b+(0,1)$.
Then, by applying  the right move to a position $f$,
we obtain the position $f'$ whose  $S_5$-part is
\[
 \sigma_{f'}=
 f((0,2)+b)f((0,1)+b)f((1,1)+b)f((1,2)+b)f((1,0)+b)
 =
 \sigma_f \mathtt{R},
\]
where ${\mathtt R}=(1,2)(3,4,5)\in S_5$.
Therefore, we have
$f'=f\,\hat{{\mathtt R}}$ where
$\hat{\mathtt R}=({\mathtt R},0,1)$.

Let the {\em left move} be the move which exchanges the blank vertex $b$ 
with $b+(0,2)$. Then, by applying the left move
to a position $f$, we obtain a new position $f'$
whose $S_5$-part is
\[
 \sigma_{f'} =
 f((0,2)+b)f((0,1)+b)f((1,2)+b)f((1,0)+b)f((1,1)+b)
 = \sigma_{f}{\mathtt L},
\]
where ${\mathtt L}=(1,2)(3,5,4)\in S_5$.
Thus we have 
$f'=f\,\hat{{\mathtt L}}$ where
$\hat{{\mathtt L}}=({\mathtt L}, 0, 2)$.

Let the {\em vertical move}
be the move which exchange the blank vertex $b$
with $b+(1,0)$.
Then, we have
\[
 \sigma_{f'} =
 f((1,1)+b)f((1,2)+b)f((1,0)+b)f((0,1)+b)f((0,2)+b)
 = \sigma_{f}{\mathtt V},
\]
where ${\mathtt V}=(1,4)(2,5)\in S_5$.
Thus we have $f'=f\,\hat{{\mathtt V}}$ where
$\hat{\mathtt V}=({\mathtt V}, 1, 0)$.

Thus the state graph ${\rm puz}(\Gamma_{2,3})$
of the cylindrical $5$-puzzle can be expressed as a Cayley graph
\[
 {\rm puz}={\rm puz}(\Gamma_{2,3})={\rm Cay}(G, \{\hat{\mathtt{\mathtt L}}, \hat{\mathtt{\mathtt R}}, \hat{\mathtt{\mathtt V}}\}).
\]

\section{Hamiltonian path and cycle}
In this section, 
we show a Hamiltonian path and a Hamiltonian cycle
in ${\rm puz}(\Gamma_{2,3})$ explicitly.
To prove that the path we will construct is a Hamiltonian path, 
we first construct a Hamiltonian cycle in
the Cayley graph ${\rm Cay}(S_5,\{{\mathtt R},{\mathtt L},{\mathtt V}\})$.

Let $H$ be a group.
For a word $w=(w_1, w_2, \ldots, w_k)\in H^{k}$,
we denote its product by $\phi_H(w)$, that is,
\[
 \phi_H(w_1,w_2,\ldots, w_k) =
 w_1 w_2 \cdots w_k \in H.
\]

\begin{lem}
 \label{lem:hamS5}
Let \( s = (s_1, s_2, \ldots, s_{12}) \in \{\mathtt{L}, \mathtt{R}, \mathtt{V}\}^{12} \) be the sequence
\begin{equation}
\label{eq:s}
s = (\mathtt{V}, \mathtt{L}, \mathtt{V}, \mathtt{R}, 
\mathtt{V}, \mathtt{L}, \mathtt{V}, \mathtt{R}, \mathtt{V}, \mathtt{R}, \mathtt{V}, \mathtt{L}), 
\end{equation}
and let \( s' = (s_{13}, s_{14}, \ldots, s_{24}) \) be its perturbation defined by
\begin{equation}
 \label{eq:s'}
s_{12+i} = s_i \quad \mbox{for } i = 1, 2, \ldots, 11, \qquad s_{24} = \mathtt{R}.
\end{equation}
 Define $w$ as the five-fold repetition of $(s_1,s_2,\ldots, s_{24})$,
 that is,
 \begin{equation}
  \label{eq:w}
 w  = (s_1,\ldots, s_{24}, s_1, \ldots, s_{24}, s_1, \ldots, s_{24},
 s_1, \ldots, s_{24}, s_1, \ldots, s_{24}
 )
 \in S_5^{120}.
 \end{equation}
 Then the sequence 
 \[
  C = ({\rm id}, g_1, g_2, \ldots, g_{120}) \in S_5^{121}
 \]
 is a Hamiltonian cycle in the
 Cayley graph ${\rm Cay}(S_5,\{{\mathtt L}, {\mathtt R}, {\mathtt V}\})$, where
 $g_i=\phi_{S_5}(w_1,w_2,\ldots, w_i)$.
\end{lem}

\begin{proof}
 As a first step, we need to verify that the elements $g_{i}$ are distinct for $i=0,1,\ldots,24$.
 This can be done by checking one-by-one. (See Table $\ref{tab:24}$.)
\begin{center}
 \begin{table}[H]
  \begin{center}
\begin{tabular}{c|c|c|c}
$i$ & $w_i$ & $g_i$ &  $\rho_i$ \\
 \hline
1  &   $\mathtt{V}$ &  45312 & 12453 \\
2  &   $\mathtt{L}$ &  54231 & 13524 \\
3  &   $\mathtt{V}$ &  31254 & 12435 \\
4  &   $\mathtt{R}$ &  13542 & 13542 \\
5  &   $\mathtt{V}$ &  42513 & 13254 \\
6  &   $\mathtt{L}$ &  24351 & 12543 \\
7  &   $\mathtt{V}$ &  51324 & 14253 \\
8  &   $\mathtt{R}$ &  15243 & 15243 \\
9  &   $\mathtt{V}$ &  43215 & 14352 \\
10 &   $\mathtt{R}$ &  34152 & 15234 \\
11 &   $\mathtt{V}$ &  52134 & 15423 \\
12 &   $\mathtt{L}$ &  25413 & 14235 \\
\end{tabular}
~~
\begin{tabular}{c|c|c|c}
$i$ & $w_i$ & $g_i$ &  $\rho_i$ \\
 \hline
13 &   $\mathtt{V}$ &  13425 & 13425 \\
14 &   $\mathtt{L}$ &  31542 & 12354 \\
15 &   $\mathtt{V}$ &  42531 & 13245 \\
16 &   $\mathtt{R}$ &  24315 & 12534 \\
17 &   $\mathtt{V}$ &  15324 & 15324 \\
18 &   $\mathtt{L}$ &  51432 & 14325 \\
19 &   $\mathtt{V}$ &  32451 & 14532 \\
20 &   $\mathtt{R}$ &  23514 & 15432 \\
21 &   $\mathtt{V}$ &  14523 & 14523 \\
22 &   $\mathtt{R}$ &  41235 & 15342 \\
23 &   $\mathtt{V}$ &  35241 & 13452 \\
24 &   $\mathtt{R}$ &  53412 & 12345 \\
\end{tabular}
  \end{center}
  \caption{$g_{i} \mbox{ for } i=1,2,\ldots,24$}
  \label{tab:24}
 \end{table}
\end{center}
Then each $g_i$ can be uniquely expressed as
\[
 g_i=g_{24}^k g_j
\]
where $i=24k+j$ with $0\leq k\leq 4$ and $1\leq j\leq 24$.
So distinctness of $g_i$ can be derived from the disjointedness
of the orbits $\langle g_{24}\rangle g_j$ for $j=1,2,\ldots, 24$.
 Since $g_{24}$ is a cyclic permutation of
order five, each orbit $\langle g_{24}\rangle g_j$ contains
exactly one pemutation $\rho_j$ such that $\rho_j(1)=1$,
which we define to be the representative of the orbit.
Such $\rho_j$ can be obtained in the following way. 
We use $g_3=31254$ for example. Since
\[
 \langle g_{24}\rangle = \left\{{\rm id},~ g_{24}=(1,5,2,3,4),~ 
 g_{24}^2=(1,2,4,5,3),~ g_{24}^3=(1,3,5,4,2),~ g_{24}^4=(1,4,3,2,5)\right\},
\]
we can compute $\rho_3$ for  $g_3=31254$ by multiplying
$g_{24}^2=(1,2,4,5,3)$ which maps $3$ to $1$, and
\[
 g_{24}^2g_3=(1,2,4,5,3)31254 = 12435.
\]
Table $\ref{tab:24}$ shows the representatives of 
the orbits $\langle g_{24}\rangle g_j$, which are distinct.
\end{proof}

\begin{thm}
\label{thm:path}
The state graph ${\rm puz}(\Gamma_{2,3})$ of the cylindrical $5$-puzzle
has a Hamiltonian path which has the following explicit expression.
Let $s$ and $s'$ be defined by $(\ref{eq:s})$ and $(\ref{eq:s'})$
respectively, and let $t$ and $t'$ be their rotations, that is,
\[
t = (\mathtt{L}, \mathtt{V}, \mathtt{R}, 
\mathtt{V}, \mathtt{L}, \mathtt{V}, \mathtt{R}, \mathtt{V}, \mathtt{R}, \mathtt{V}, \mathtt{L},\mathtt{V})
\]
and
\[
t' = (\mathtt{L}, \mathtt{V}, \mathtt{R}, 
\mathtt{V}, \mathtt{L}, \mathtt{V}, \mathtt{R}, \mathtt{V}, \mathtt{R}, \mathtt{V}, \mathtt{R},\mathtt{V}).
\]
Then define $v$ as the word obtained by deleting the last letter from
the word
\[
 (ss')^{14}\cdot s s\cdot (tt')^{15},
\]
so that $v$ has length $719$.
Then $\hat{v}$ traces a Hamiltonian path of the state graph of the cylindrical $5$-puzzle.
\end{thm}

\begin{proof}
 Let $w$ be defined by $(\ref{eq:w})$.
 Then, as we have shown in Lemma $\ref{lem:hamS5}$,
 $\phi_{S_5}(w)={\rm id}\in S_5$.
 The word $ss'$ of length $24$ contains seven ${\mathtt R}$s and 
 five ${\mathtt L}$s.
 Therefore, the ${\mathbb Z}/3{\mathbb Z}$-part
 of $\phi_G(\hat{ss'})$ is
 \[
  5\times 1 + 7\times 2 = 1\in {\mathbb Z}/3{\mathbb Z}.
 \]
 Hence the ${\mathbb Z}/3{\mathbb Z}$-part of
 $\phi_G(\hat{w})$ is $5\times{1}=2\in {\mathbb Z}/3{\mathbb Z}$.
 In the same manner, we see that
 the ${\mathbb Z}/2{\mathbb Z}$-part of
 $\phi_G(\hat{w})$ is $0\in {\mathbb Z}/2{\mathbb Z}$.
 Define 
 \[
  d=(d_1,d_2,\ldots,d_{360})=\hat{w}^3\in G^{360}.
 \]
 and
 \[
 q_i=\phi_G(d_1\cdots d_i).
 \]
 Let
 \[
 a_0 = \left({\rm id}, 0, 0\right), ~~~
 a_1 = \left({\rm id}, 1, 0\right)\in G,
 \]
 and define
 \[
  C_i = (a_i, a_iq_1, a_iq_2, \ldots, a_iq_{360})
 \]
 for $i=0,1$. Then, now we claim that $C_1$ and $C_2$ form
 a $2$-cycle cover of the state graph ${\rm puzz}$.
 To prove this, it suffices to show that each $C_i$ is
 a simple cycle and $C_1$ and $C_2$ do not intersect.
 Suppose for the sake of contradiction, 
 that there exist $i$ and $j$ such that 
 $q_i=q_j$ and $1\leq i < j\leq 360$.
 Then, since their $S_5$-parts must coincide,
 there exists an integer $k\in\{1,2\}$ such that $j=i+120k$,
 where the ${\mathbb Z}/3{\mathbb Z}$-parts do not
 coincide. Thus $C_0$ and $C_1$ are both simple cycles.
 If we have $a_0q_i=a_1q_j$, again, since 
 since their $S_5$-parts must coincide,
 there exists an integer $k\in\{1,2\}$ such that $j=i+120k$,
 where the ${\mathbb Z}/2{\mathbb Z}$-parts do not
 coincide.

 Then we splice $C_0$ and $C_1$ to form a Hamiltonian path
 in ${\rm puz}(\Gamma_{2,3})$.
 To break the cycle of length $360$,
 we modify the last letter $d_{360}={\mathtt R}$ of $d$ to ${\mathtt L}$.
 Then
 \[
  \phi_G(d_1,d_2,\ldots,d_{359},\hat{{\mathtt L}}) =
 a_0q_{360}\hat{\mathtt R}^{-1}\hat{{\mathtt L}}
 =
 \hat{{\mathtt L}}^2 
 =(12453,0,1) = (g_1, 0, 1).
 \]
 Therefore, setting $v'=(v_1,\ldots, v_{360})=(tt')^{15}$,
 we obtain
 \[
 (
 \hat{{\mathtt L}}^2,
 \hat{{\mathtt L}}^2\phi_G(v_1),
 \hat{{\mathtt L}}^2\phi_G(v_1,v_2),
 \ldots,
 \hat{{\mathtt L}}^2\phi_G(v_1,\ldots,v_{360})
 ).
 \]
 is a simple cycle in ${\rm puz}(\Gamma_{2,3})$
 starting from and terminating
 at $\hat{{\mathtt L}}^2$, which is a cyclic shift of $C_1$
 and disjoint with $C_0$.
 Thus, 
 the word obtained by removing the last letter from the word
 $(ss')^{14}ss(tt')^{15}$ traces out a Hamiltonian path
 of ${\rm puz}(\Gamma_{2,3})$.
\end{proof}

Although the verification of the following theorem
can be performed algorithmically,
we provide an explicit list of the 720 states (Supplement S1) 
as a formal certificate of Theorem $\ref{thm:cycle}$.
\begin{thm}
\label{thm:cycle}
 The state graph ${\rm puz}\left(\Gamma_{2,3}\right)$
 has a Hamiltonian cycle, which is given by the 
 word $\hat{c}$ on $\{\hat{{\mathtt L}}, \hat{{\mathtt R}}, \hat{{\mathtt V}}\}$, where
 \[
 c = \mathtt{R} \mathtt{V} \mathtt{L} \mathtt{V} \mathtt{R} \mathtt{V} \mathtt{L} \mathtt{V} \mathtt{R} \mathtt{R} \mathtt{R} \mathtt{V} \mathtt{L} \mathtt{V} \mathtt{R} \mathtt{V} \mathtt{R} \mathtt{V} \mathtt{L} \mathtt{V} \mathtt{L} \mathtt{V} \mathtt{R} \mathtt{V} \mathtt{L} \mathtt{V} \mathtt{L} \mathtt{L} \mathtt{V} \mathtt{R} \mathtt{V} \mathtt{R} \mathtt{V} \mathtt{L} \mathtt{L} \mathtt{L} \mathtt{V} \mathtt{L} \mathtt{V} \mathtt{R} \mathtt{R} \mathtt{R} \mathtt{V} \mathtt{R} \mathtt{V} \mathtt{R} \mathtt{V} \mathtt{R}.
 \]
\end{thm}
\begin{proof}
 Note that any repetition of the word $c$ does not trace a Hamiltonian cycle
 of ${\rm Cay}(S_5, \{{\mathtt L}, {\mathtt R}, {\mathtt V}\})$.
 The proof of this theorem proceeds in essentially the same way as that of Lemma $\ref{lem:hamS5}$, 
 except that a larger table of size $48$ is needed. 
 Table $\ref{tab:large}$ shows $g_i=\phi_G(\hat{c}_1,\ldots,\hat{c}_i)$ for $i=1,2,\ldots,48$ and
 the representative $\rho_i$ of $g_i$'s orbit $\langle g_{48}\rangle g_i$, which is defined
 as follows:
 As can be verified in Table $\ref{tab:large}$,
 \[
 g_{48}=\phi_G(\hat{c})=(35214, 0, 1) \in G = S_5 \times {\mathbb Z}/2{\mathbb Z} \times {\mathbb Z}/3{\mathbb Z}.
 \]
 Since $35214=(1,3,2,5,4)$ has order five and $1\in {\mathbb Z}/3{\mathbb Z}$ has order three,
 $g_{48}$ generates the cyclic subgroup $\langle g_{48}\rangle$ of order $15$.
 The representative $\rho_i=(\sigma_i, x_i, y_i)$ is defined to be the
 unique element which satisfies $\sigma_i(1)=1$ and $y_i=0$.
 It can be verified that the representatives $\rho_i$ are distinct.
\begin{center}
 \begin{table}[H]
  \begin{center}
   \begin{tabular}{c|c|c|c}
    $i$ & $\hat{c}_i$ & $g_i$ & $\rho_i$\\
    \hline
   $1 $ &   {\tt R} &  $(21453,0,1)$ & $(15234,0,0)$\\
   $2 $ &   {\tt V} &  $(53421,1,1)$ & $(15342,1,0)$\\
   $3 $ &   {\tt L} &  $(35142,1,0)$ & $(12453,1,0)$\\
   $4 $ &   {\tt V} &  $(42135,0,0)$ & $(15324,0,0)$\\
   $5 $ &   {\tt R} &  $(24351,0,1)$ & $(12435,0,0)$\\
   $6 $ &   {\tt V} &  $(51324,1,1)$ & $(12543,1,0)$\\
   $7 $ &   {\tt L} &  $(15432,1,0)$ & $(15432,1,0)$\\
   $8 $ &   {\tt V} &  $(32415,0,0)$ & $(13542,0,0)$\\
   $9 $ &   {\tt R} &  $(23154,0,1)$ & $(14532,0,0)$\\
   $10$ &   {\tt R} &  $(32541,0,2)$ & $(13254,0,0)$\\
   $11$ &   {\tt R} &  $(23415,0,0)$ & $(14253,0,0)$\\
   $12$ &   {\tt V} &  $(15423,1,0)$ & $(15423,1,0)$\\
   $13$ &   {\tt L} &  $(51342,1,2)$ & $(12534,1,0)$\\
   $14$ &   {\tt V} &  $(42351,0,2)$ & $(15243,0,0)$\\
   $15$ &   {\tt R} &  $(24513,0,0)$ & $(12354,0,0)$\\
   $16$ &   {\tt V} &  $(13524,1,0)$ & $(13524,1,0)$\\
   $17$ &   {\tt R} &  $(31245,1,1)$ & $(14352,1,0)$\\
   $18$ &   {\tt V} &  $(45231,0,1)$ & $(14523,0,0)$\\
   $19$ &   {\tt L} &  $(54123,0,0)$ & $(13245,0,0)$\\
   $20$ &   {\tt V} &  $(23154,1,0)$ & $(14532,1,0)$\\
   $21$ &   {\tt L} &  $(32415,1,2)$ & $(13542,1,0)$\\
   $22$ &   {\tt V} &  $(15432,0,2)$ & $(15432,0,0)$\\
   $23$ &   {\tt R} &  $(51324,0,0)$ & $(12543,0,0)$\\
   $24$ &   {\tt V} &  $(24351,1,0)$ & $(12435,1,0)$\\
   \end{tabular}~~
   \begin{tabular}{c|c|c|c}
    $i$ & $\hat{c}_i$ & $g_i$ & $\rho_i$\\
    \hline
   $25$ &   {\tt L} &  $(42135,1,2)$ & $(15324,1,0)$\\
   $26$ &   {\tt V} &  $(35142,0,2)$ & $(12453,0,0)$\\
   $27$ &   {\tt L} &  $(53214,0,1)$ & $(15423,0,0)$\\
   $28$ &   {\tt L} &  $(35421,0,0)$ & $(12534,0,0)$\\
   $29$ &   {\tt V} &  $(21435,1,0)$ & $(15243,1,0)$\\
   $30$ &   {\tt R} &  $(12354,1,1)$ & $(12354,1,0)$\\
   $31$ &   {\tt V} &  $(54312,0,1)$ & $(13524,0,0)$\\
   $32$ &   {\tt R} &  $(45123,0,2)$ & $(14352,0,0)$\\
   $33$ &   {\tt V} &  $(23145,1,2)$ & $(14523,1,0)$\\
   $34$ &   {\tt L} &  $(32514,1,1)$ & $(13245,1,0)$\\
   $35$ &   {\tt L} &  $(23451,1,0)$ & $(14235,1,0)$\\
   $36$ &   {\tt L} &  $(32145,1,2)$ & $(13452,1,0)$\\
   $37$ &   {\tt V} &  $(45132,0,2)$ & $(14325,0,0)$\\
   $38$ &   {\tt L} &  $(54213,0,1)$ & $(13425,0,0)$\\
   $39$ &   {\tt V} &  $(13254,1,1)$ & $(13254,1,0)$\\
   $40$ &   {\tt R} &  $(31542,1,2)$ & $(14253,1,0)$\\
   $41$ &   {\tt R} &  $(13425,1,0)$ & $(13425,1,0)$\\
   $42$ &   {\tt R} &  $(31254,1,1)$ & $(14325,1,0)$\\
   $43$ &   {\tt V} &  $(54231,0,1)$ & $(13452,0,0)$\\
   $44$ &   {\tt R} &  $(45312,0,2)$ & $(14235,0,0)$\\
   $45$ &   {\tt V} &  $(12345,1,2)$ & $(12345,1,0)$\\
   $46$ &   {\tt R} &  $(21453,1,0)$ & $(15234,1,0)$\\
   $47$ &   {\tt V} &  $(53421,0,0)$ & $(15342,0,0)$\\
   $48$ &   {\tt R} &  $(35214,0,1)$ & $(12345,0,0)$\\
   \end{tabular}
  \end{center}
  \caption{$g_i$ for $i=1,2,\ldots, 48$ and their $\langle g_{48}\rangle$-orbit
  representatives}
  \label{tab:large}
 \end{table}
\end{center}
 
\end{proof}

\section{Quotient-Lifting method}
This section shows a heuristic to generate  
Hamiltonian cycles and paths.
The Hamiltonian path and the Hamiltonian cycle 
shown in the previous section
can be found using this heuristic.

Let $H$ be a group and $S$ be a generating set of $H$.
Then finding a Hamiltonian cycle in the Cayley graph $X={\rm Cay}(H, S)$
is generally a computationally difficult problem.
The naive backtracking does not work for general large graphs  \cite{ruskey1995hamiltonicity}.
Our heuristics is simple.

\begin{enumerate}
 \item Choose an {\em appropriate} subgroup $K$ of the automorphism group ${\rm Aut}(X)$
       of $X$, which acts freely on $V(X)$ from left,
       and let $\pi$ be the covering map  (or the natural projection)
       from $X$ to
       the quotient graph $Y=K\backslash X$.
 \item If we can find a Hamiltonian cycle in $Y$, lift it
       back to $X$ using $\pi^{-1}$.
 \item If the cycles in the lift  do not form a Hamiltonian cycle of $X$,
       try to splice them to form a Hamiltonian cycle or a Hamiltonian path.
\end{enumerate}

We do not know how to find an {\em appropriate} subgroup $K$ for general $X$.
Here we show two successful examples in $X={\rm puz}(\Gamma_{2,3})$.

Let $K$ be a subgroup of $G=S_5\times {\mathbb Z}/2{\mathbb Z} \times {\mathbb Z}/3{\mathbb Z}$.
Then $K$ naturally acts on the vertex set $V(X)$ as a subgroup
of ${\rm Aut}(X)$ and the vertex set $V(Y)$
is the cosets $\{Kg \,|\, g\in G\}$.
There is an edge from $Ka$ to $Kb$
if and only if there exists $\xi \in \{\hat{\mathtt L}, \hat{\mathtt R}, \hat{\mathtt V}\}$,
such that $Ka\xi = Kb$.

Let the subgroup be
\[
 K_0=\langle(1,2,3,4,5)\rangle\times {\mathbb Z}/2{{\mathbb Z}} \times {\mathbb Z}/3{\mathbb Z}
 \subset G.
\]
Then, the quotient graph  $K_0\backslash{X}$ has $24$ vertices,
in which finding a Hamiltonian cycle is computationally feasible.
We identify sequences over  $\{{\mathtt L}, {\mathtt R}, {\mathtt V}\}$
under the group action generated by
rotation, reversal, and reflection defined by
\begin{itemize}
 \item The {\em rotation} of $x=(x_1,\ldots, x_k)$ is $(x_2, x_3,\ldots, x_k, x_1)$.
 \item The {\em reversal} of $x=(x_1,\ldots, x_k)$ is $(x_k,\ldots, x_2, x_1)$.
 \item The {\em reflection} of $x\in \{{\mathtt L}, {\mathtt R}, {\mathtt V}\}^k$ is the word obtained from $x$ by interchanging ${\mathtt L}$ and ${\mathtt R}$ in $x$.
\end{itemize}
Then $K_0\backslash X$ has only four Hamiltonian cycles listed in Table $\ref{tab:s5hamilton}$,
each of which exhibits an 
interesting symmetric pattern, and 
whose five-fold repetitions
trace Hamiltonian cycles
of ${\rm Cay}(S_5,\{{\mathtt L} , {\mathtt R}, {\mathtt V}\})$.
This suggests the presence of
a combinatorial grammar underlying Hamiltonian paths.
Two of the Hamiltonian cycles or $K_0\backslash{X}$
are lifted to form $2$-cycle covers of $X$, and the rest are
lifted to form $6$-cycle covers of $X$.

\begin{table}[H]
 \begin{center}
  \begin{tabular}{c|c}
   Hamiltonian cycle of $K_0\backslash{X}$ & lift to $X$\\
   \hline
   $\mathtt{VRVLVRVLVLVR ~ VRVLVRVLVLVL}$ & $2$-cycle cover \\
   $\mathtt{VLVLVLVRVLVL ~ VRVRVRVLVRVR}$ & $2$-cycle cover \\
   $\mathtt{LLLLLVLVRRRV ~ LLLVRVRRRRRV}$ & $6$-cycle cover \\
   $\mathtt{LLLLLVRRRRRV ~ LVRRRVLLLVRV}$ & $6$-cycle cover \\
  \end{tabular}
  \caption{Hamiltonian cycles of $K_0\backslash{X}$ and their lifts to $X$}
  \label{tab:s5hamilton}
 \end{center}
\end{table}

For the subgroup $K_1$ defined by
\[
 K_1=\langle(1,2,3,4,5)\rangle\times \{0\} \times {\mathbb Z}/3{\mathbb Z},
\]
the quotient graph $K_1\backslash X$ has $48$ vertices,
in which finding a Hamiltonian cycle is still computationally feasible.
By lifting a Hamiltonian cycle in $K_1\backslash X$, we have obtained the Hamiltonian cycle of $X$,
which is shown in Theorem $\ref{thm:cycle}$.
The quotient $K_1\backslash X$ and ${\rm puz}(\Gamma_{2,3})$
are illustrated in Figure $\ref{fig:QLM}$.

\begin{center}
 \begin{figure}[H]
  \begin{center}
   \includegraphics[width=6cm]{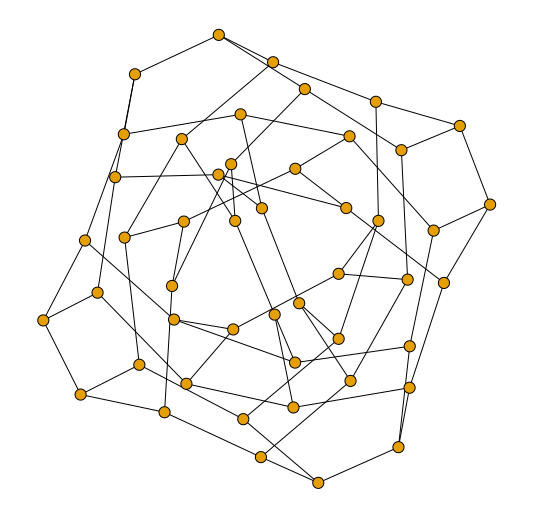}
   \includegraphics[bb=0 0 758 829,clip,width=6cm]{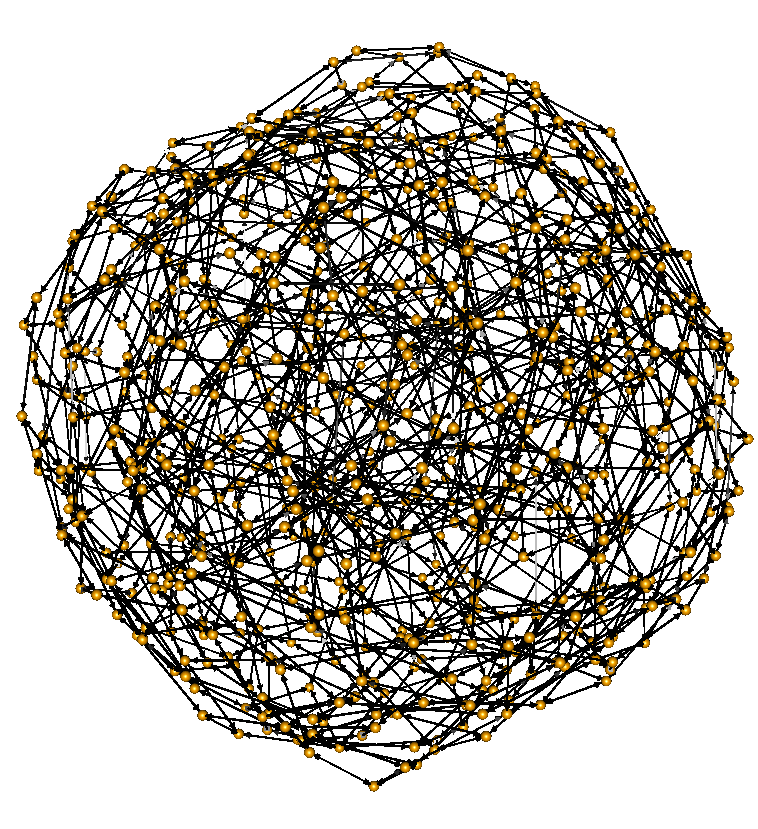}
  \end{center}
  \caption{
  A Hamiltonian cycle on the quotient graph $K_1\backslash X$
  of order $48$ (left side) lifts to a single Hamiltonian cycle of length $720$ in the full state graph 
  ${\rm puz}(\Gamma_{2,3})$ (right side).
  This visualization illustrates the essence of the quotient-lifting method.}
  \label{fig:QLM}
 \end{figure}
\end{center}

\begin{rmk}
The remarkable point is that the lifted cycle remains single despite being $15$ times longer.
This can be explained by the fact that the monodromy permutation induced on the fiber is a $15$-cycle.
\end{rmk}

\subsubsection*{Future work}
The same quotient-lifting framework can naturally be extended to larger puzzles such as the 7-puzzle, where the monodromy becomes more intricate.

\section*{Acknowledgement}
The author would like to express sincere gratitude to the anonymous reviewers for their careful reading of the manuscript and for providing many valuable suggestions that helped to improve the paper.


\end{document}